\newcommand{\comment}[1]{}
\newcount\hh
\newcount\mm
\mm=\time
\hh=\time
\divide\hh by 60
\divide\mm by 60
\multiply\mm by 60
\mm=-\mm
\advance\mm by \time
\def\hhmm{\number\hh:\ifnum\mm<10{}0\fi\number\mm}

\documentclass{article}

\usepackage{amssymb, amsthm,amsmath}
\usepackage{pictexwd, dcpic}
\usepackage[]{hyperref}
\hypersetup{
colorlinks=true,
linkcolor=black,
}
\usepackage{a4wide} 
\usepackage{fancyhdr}

\theoremstyle{definition} 
\newtheorem{theorem}{Theorem}
\newtheorem{prop}[theorem]{Proposition}
\newtheorem{lemma}[theorem]{Lemma}
\newtheorem*{defi}{Definition}
\newtheorem*{exam}{Example}
\newtheorem*{ack}{Acknowledgement}
\newtheorem{cor}[theorem]{Corollary}
\newtheorem{remark}[theorem]{Remark}

\theoremstyle{remark} 
\newtheorem*{rem}{Remark} 
\newtheorem*{nota}{Notation}

\newcommand\brac[1]{\langle #1 \rangle}
\newcommand\sep{\,:\,}

\newcommand\st {\mathrel{\ooalign{$\,\backepsilon$\cr\lower .7pt\hbox{\kern 1pt$-\,$}}}}
\newcommand\qf[1]{\mathrm{Quot}(#1)}	
\newcommand\spec{\mathrm{Spec}\,}	
\newcommand\sper{\mathrm{Sper}\,}	
\newcommand\supp{\mathrm{supp}}	
\renewcommand\O{\mathcal{O}}

\newcommand\Z{\mathbb{Z}}
\newcommand\Q{\mathbb{Q}}

\newcommand{\N}{\mathbb{N}}
\newcommand{\p}{\mathfrak{p}}
\newcommand\wo[1]{\backslash{\{#1\}}}

\newcommand*{\lhrarrow}{\ensuremath{\lhook\joinrel\relbar\joinrel\rightarrow}}
\newcommand*{\thrarrow}{\twoheadrightarrow}
\renewcommand\mod{\,\mathrm{mod}\,}
\newcommand{\ic}{\mathrm{ic}}

\begin{document}

\title{Uniqueness of real closure $*$ of Baer regular rings}
\makeatletter
\let\mytitle\@title
\makeatother
\author{Jose Capco\\
\href{mailto:capco@fim.uni-passau.de}{\small{capco@fim.uni-passau.de}}\\
\emph{\small{Universit\"at Passau, Innstr. 33, 94032 Passau, Germany}} }
\date{}
\thispagestyle{empty}
\maketitle

\pagestyle{fancy}
\fancyhead[R]{\mytitle}
\fancyhead[L]{J. Capco}
\tolerance=500

\begin{abstract}
It was pointed out in my last paper that there are rings whose real closure $*$ are not unique. In \cite{Capco} we also 
discussed some example of rings by which there is a unique real closure $*$ (mainly the real closed rings). 
Now we want to determine more classes of rings by which real closure $*$ is unique. The main results involve characterisations of 
domains and Baer regular rings having unique real closure $*$, and an example showing that regular rings need not be $f$-rings in order to
have a unique real closure $*$. The main objective here is to find characterisation for uniqueness of real closure $*$ for real regular rings 
that will primarily only require information of the prime spectrum and the real spectrum of the ring.
\begin{description}
\item[Mathematics Subject Classification (2000):] Primary 13J25; Secondary 13B22, 16E50
\item[Keywords:] real closed rings, real closed $*$ rings, tight/essential extension, algebraic extension, Baer von Neumann regular rings, 
$f$-rings.
\end{description}
\end{abstract}

\footnote{Supported by Deutsche Forschungsgemeinschaft and Universit\"at Passau.}

Unless otherwise stated all the rings in this paper are commutative unitary partially ordered rings (porings). If the partial
ordering is not specified then it can be assumed to be the sum of squares of the elements of the ring. The reader is assumed to know 
results in \cite{KS} Kapitel III. 

\begin{defi} 
Given a poring $A$, we usually write $A^+$ to mean the partial ordering of $A$ that is being considered. If $B$ is an
over-ring of a $A$ that is also a poring and if it has the property that $B^+\supset A^+$ we shall call it an \emph{over-poring} of $A$
\end{defi}

\begin{rem}
Throughout when we deal with a poring $A$ the Baer hull $B(A)$ (see \cite{Capco} Proposition 2) of the poring has the partial ordering 
$$B(A)^+:=\{\sum_{i=1}^n b_i^2a_i \sep n\in \N, b_i\in B(A), a_i\in A^+ \textrm{ for } i=1,\dots,n\} $$
\end{rem}

In \cite{SV} a real closure $*$ of a real ring was defined to be a real ring
that is essential integral extension of the ring that is real closed $*$. One notices therefore that the definition is rather ring theoretic
 and the partial ordering of the ring is not even being questioned. This may lead to several confusion and in fact it may be necessary that 
we allow the partial ordering of our rings to play a role in the definition. Imagine for instance $\Q(\sqrt 2)$, this is a real field and 
it is well-known (see \cite{epi} Example 2.1) that there are two total orderings of it , one is say $\alpha$ by which 
$\sqrt 2 \in \alpha$ (thus $\sqrt 2$ is positive)
and the other is say $\beta$ with $-\sqrt 2 \in \beta$ (thus $-\sqrt 2$ is positive). Set $A:=\Q(\sqrt 2)$, then if $\alpha\cap\beta$
is the partial ordering of $A$, we see that from the definition of \cite{SV} $A$ has two possible real closure $*$ which are not 
contained in each other (and are real closed fields).  Thus is indeed reasonable to require that the extension be a poring 
extension. 

\begin{defi}
Let $A$ be a real ring, then \underline{a} \emph{real closure $*$} of $A$ is a real ring $B$ such that
\begin{itemize}
\item $B$ is real closed $*$
\item $B$ is an essential and integral extension of $A$
\item $A^+\subset B^+$
\end{itemize}
\end{defi}

If we require this then $Q(\sqrt 2)$ with $\alpha$ or $\beta$ as partial ordering have a unique real closure $*$ but if its partial ordering is
$\alpha\cap\beta$ then the real closure $*$ is not unique. This is a definition that has been used (but not mentioned) in \cite{Capco}
, for instance in \cite{Capco} Theorem 23 we required that $\rho(A)$ (the real closure of a real ring $A$) to be essential over $A$. 
Essentiality of $\rho(A)$ is very much dependent on the partial ordering of $A$, as the construction of the real closure of $A$ is dependent
on the partial ordering of $A$ and it is therefore sound to have the partial ordering play a role in our definition of real closure $*$.

It is best that we define what we actually mean by the \emph{uniqueness of real closure $*$}. 

\begin{defi}
Let $A$ be a real ring
\begin{itemize}
\item[1.)] If $B$ and $C$ are over-rings of $A$, we say that they are \emph{$A$-isomorphic} iff there is an isomorphism of 
porings
$$f:B\stackrel{\sim}{\longrightarrow} C$$
such that $fi_B = i_C$, where $i_B:A\hookrightarrow B$ and $i_C:A\hookrightarrow C$ are the canonical poring injections. 
In other words $B$ and $C$ are $A$-isomorphic if there is a poring isomorphism between them which \emph{fixes} $A$. We may
use the notation 
$$B\cong_A C$$ 
to mean  $B$ and $C$ are $A$-isomorphic. Similarly one defines $A$-isomorphic for porings, if it is not stated we usually do
mean $A$-isomorphic as porings (i.e. the homomorphisms preserve order).
\item[2.)] We say that $A$ has a \emph{unique real closure $*$} iff the following holds
\begin{itemize}
\item[$\bullet$] There is a real closed $*$ ring $B$ which is an essential and integral extension of $A$
\item[$\bullet$] If there is another real closed $*$ ring, say $C$, which is essential and integral over $A$ then $C$ and $B$ are
$A$-isomorphic.
\end{itemize}
If $A$ does not satisfy the above, then we say that \emph{$A$ has no unique real closure $*$}.
\end{itemize}
\end{defi}

We note that we are particularly dealing with poring morphisms. So unless otherwise stated, all the maps and diagrams 
are in the category of porings.

\begin{nota}
Let $A$ be a real ring and $B$ be an over-poring of it. Then by $\ic(A,B)$ we mean the integral closure of $A$ in $B$ and with partial
ordering 
$$\ic(A,B)^+=B^+\cap \ic(A,B)$$ 
Thus we consider $\ic(A,B)$ as an object in the category of porings.
\end{nota}

\begin{remark}
By an essential extension or integral extension, unless otherwise stated. We mean an essential/integral extension of a 
poring which extends the partial ordering as well.
\end{remark}

\begin{lemma}\label{ie_of_id}
Let $A$ and $B$ be integral domains and $A\subset B$ as rings, then
\renewcommand{\labelenumi}{\roman{enumi}.}
\begin{enumerate}
\item If $B$ is integral over $A$, then $\qf B$ is integral over $\qf A$.
\item If $B$ is an essential extension of $A$ and $A$ is integrally closed in $B$ then $\qf A$ is integrally closed in $\qf B$
\end{enumerate}
\renewcommand{\labelenumi}{\arabic{enumi}.}
\end{lemma}
\begin{proof} i. The first part is a basic result in field theory, the set of elements in $\qf B$ that is algebraic (or integral, as we 
are in fields) over $\qf A$ form a field, call it $K$. Then $B\subset K$ since $B$ is integral over $A$. Moreover $B\subset K\subset \qf B$ 
implies that $K=\qf B$. Thus $\qf B$ is integral over $\qf A$.

ii. Let $\frac{b}{c}$ be an element of $\qf B$, with $b,c\in B$ and $c\neq 0$. Because $B$ is essential over $A$, without loss 
of generality we may assume that $c\in A$. Suppose furthermore that $\frac{b}{c}$ is an 
integral element of $\qf A$. Then there is a polynomial 
$$f(T)=T^n + \sum_{i=0}^{n-1} T^i \frac{a_i}{x} \in \qf A [T] $$
with $a_i \in A$ and $x \in A\wo{0}$ and $\frac{b}{c}$ a zero of $f$. So 
$$f(b/c)=\left(\frac{b}{c}\right)^n + \sum_{i=0}^{n-1} \frac{b^i a_i}{xc^i} = 0$$
Now multiply $f(b/c)$ by $(cx)^n$ (which isnt $0$, since $c$ and $x$ arent), then 
$$c^nx^nf(b/c)= (bx)^n + \sum_{i=0}^{n-1} (bx)^i c^{n-i}x^{n-i-1}a_i = 0 $$
and therefore $bx\in B$ is a zero of 
$$T^n + \sum_{i=0}^{n-1} b_i T^i  \in A[T] $$
where $b_i := c^{n-i}x^{n-i-1}a_i \in A$ for $i=0,\dots,n-1$.
But $A$ is integrally closed in $B$, so $bx \in A$. This implies that 
$$b/c=bx/cx\in \qf A$$
\end{proof}

\begin{remark}
Note that in (ii) of above, essentiality is required. Since for instance $\Z$ is integrally closed in $\Z[T,T\sqrt 2]$ and 
their respective quotient fields are $\Q$ and $\Q(\sqrt2,T)$. Yet it is clear that $\Q$ isn't integrally closed in $\Q(\sqrt2,T)$.
\end{remark}

\begin{theorem}
A real integral domain has a unique real closure $*$ iff its quotient field has only one total ordering
\end{theorem}
\begin{proof}
"$\Leftarrow$" Let $A$ be a partially ordered integral domain such that $\qf A$ has only one total ordering. Let $B_1$ and $B_2$ be two
 real closure $*$ of $A$. Then $B_1$ and $B_2$ are real closed $*$ rings and are integral and essential extensions of $A$. By Lemma 2 in 
\cite{SV} $B_1$ and $B_2$ are integral domains (Lemma 2 in \cite{SV} states that the essential extensions of integral domains 
are integral domains). So we have a commutative diagram of porings below 
$$
\begindc{\commdiag}[1]
\obj(0,60)[A]{$A$}
\obj(100,60)[B]{$B_i$}
\obj(0,0)[C]{$\qf A$}
\obj(100,0)[D]{$\qf{B_i}$}
\obj(170,30){$i=1,2$}
\mor{A}{B}{}[1,3]
\mor{B}{D}{}[1,3]
\mor{A}{C}{}[-1,3]
\mor{C}{D}{}[1,3]
\enddc
$$
with the lower horizontal map being an integral extension of $\qf{B_i}$ (by the Lemma above). 

For $i=1,2$, $B_i$ is a real closed $*$ ring and so by \cite{SV} Proposition 2 $\qf{B_i}$ is a real closed field and $B_i$ is 
integrally closed in it. So $B_i$ is the integral closure of $A$ in $\qf{B_i}$ (upto $A$-isomorphism) and therefore by the Lemma above
$\qf A \hookrightarrow \qf{B_i}$ is an algebraic extension of $\qf A$ which is real closed. But we assumed that the $\qf A$ has only one 
point in its real spectrum and so up to $A$-isomorphism we can set 
$$R:=\qf{B_1}=\qf{B_2}$$
and $B_i$ being the integral closure of $A$ in $R$ (because $R$ is real closed $*$ and essential over $A$) we get upto $A$-isomorphism that
$B_1=B_2$.

"$\Rightarrow$" Let $A$ be a real integral domain with a unique real closure $*$. Suppose that $\qf A$ has more than one total ordering. 
Then there are real closed fields $K_1,K_2$ not containing each other (upto $A$-isomorphism) and yet algebraic (integral) extensions 
of $\qf A$. 

Now for $i=1,2$ let $A_i$ be the integral closure of $A$ in $K_i$. Thus we have the commutative diagram of porings below
$$
\begindc{\commdiag}[1]
\obj(0,70)[A]{$A$}
\obj(80,70)[B]{$A_i$}
\obj(80,0)[C]{$\qf{A_i}$}
\obj(180,70)[D]{$K_i$}
\obj(250,35){$i=1,2$}
\mor{A}{B}{}[1,3]
\mor{B}{C}{}[1,3]
\mor{B}{D}{}[1,3]
\mor{C}{D}{}[1,3]
\enddc
$$
And so by Lemma \ref{ie_of_id} $\qf{A_i}$ is integrally closed in $K_i$, thus $\qf{A_i}$ is a real closed field (see \cite{KS} Korollar of
p.17) with $A_i$ integrally closed in $\qf{A_i}$. We can thus conclude that $A_i$ is a real closed $*$ integral domain which is an essential
 and integral extension of $A$. So by uniqueness of real closure $*$, $A_1=A_2$ upto $A$-isomorphism and we can also set
$$R:=\qf{A_1}=\qf{A_2}$$
we thus have 
$$\qf A\hookrightarrow R\hookrightarrow K_i\qquad i=1,2$$
with $K_1,K_2$ being algebraic extensions of $\qf A$ and being real closed. But this only implies that $K_1=K_2=R$ upto $\qf A$-isomorphism.
This contradicts our original assumption for $K_1$ and $K_2$.
\end{proof}

\begin{prop}
Given a real domain $A$, $B$ is a real closure $*$ of $A$ iff
$B$ is the integral closure of $A$ in some real closed field $K$ algebraic over $\qf A$ with $K^+$ extending $\qf A^+$, and 
$B^+=B\cap K^+$.
\end{prop}
\begin{proof}
"$\Rightarrow$"
Let $B$ be a real closure $*$ of $A$, we then have the following poring monomorphism $A\hookrightarrow B$, with $B$ being integral
over $A$. Because of \cite{SV} Proposition 2, $\qf B$ is a real closed field and $B$ is integrally closed in it. We thus have the following
embeddings of partially ordered integral domains
$$A\lhrarrow B \lhrarrow \qf B$$
and $B$ is thus the integral closure of $A$ in $\qf B$ (as $\qf B$ is trivially essential over both $B$ and $A$). Moreover because $B$ 
is integral over $A$, by Lemma \ref{ie_of_id}(i) $\qf B$ is 
algebraic over $\qf A$. The part regarding the partial orderings are automatically satisfied, as we have throughout been dealing with 
poring monomorphisms the canonical partial orderings of $\qf A,\qf B$.

\vspace{5mm}
"$\Leftarrow$"
Let $K$ be a real closed field and suppose it satisfies the sufficiency conditions of the Proposition. Thus we assume that 
$$B=\ic(A,K), \quad B^+ =B\cap K^+$$
We therefore have the following poring embeddings
$$A\subset B\subset \qf B\subset K$$
$B$ is integrally closed in $K$ and $K$ is obviously essential over $B$, this implies (by Lemma \ref{ie_of_id}(ii)) that 
$\qf B$ is algebraically closed in $K$. By \cite{KS} Korollar p.17, $\qf B$ is therefore a real closed field. Thus by \cite{SV} Proposition
2, $B$ is a real closed $*$ ring.
\end{proof}

\begin{prop}\label{ic-inrc}
Let $A$ be a real von Neumann regular ring and let $B$ be an essential extension of $A$ such that it is a von Neumann regular 
real closed $*$ ring. Then one real closure $*$ of $A$ would be the integral closure of it in $B$.
\end{prop}
\begin{proof}
Denote $\bar A$ to be the integral closure of $A$ in $B$. By Corollary 1.10 in \cite{raphael}, $\bar A$ is a von Neumann regular ring. 
Moreover $B$ is essential over $\bar A$. Thus $Q(B)$ may be regarded as an over-ring of $Q(\bar A)$ (see \cite{Storrer} Satz 10.1). And the diagram 
of porings below 
$$
\begindc{\commdiag}[1]
\obj(0,60)[A]{$\bar A$}
\obj(100,60)[B]{$B$}
\obj(0,0)[C]{$Q(\bar A)$}
\obj(100,0)[D]{$Q(B)$}
\mor{A}{B}{}[1,3]
\mor{B}{D}{}[1,3]
\mor{A}{C}{}[-1,3]
\mor{C}{D}{}[1,3]
\enddc
$$
is commutative. Since $B$ is real closed $*$, it contains all the idempotents of $Q(B)$ 
(thus also all idempotents of $Q(\bar A)$), and since $\bar A$ is integrally closed in $B$ it contains all the idempotents of $Q(\bar A)$. 
In other words $\bar A$ is Baer.

By Proposition 14 in \cite{Capco} $B$ is also a real closed ring. And because $\bar A$ is a von Neumann regular integrally closed
subring of $B$ (which is an $f$-ring), $\bar A$ must be a sub-$f$-ring of $B$ (see Lemma 13 in \cite{Capco}). Theorem 9 in 
\cite{Capco} tells us that an integrally closed von Neumann regular sub-$f$-ring of a real closed ring is itself real closed. Thus 
$\bar A$ is Baer and real closed and from the characterization of real closed $*$ von Neumann regular rings found in \cite{Capco} Theorem 15
this means that $\bar A$ is real closed $*$. 
\end{proof}

\begin{defi}
If a commutative unitary ring $A$ has and over-ring $B$. Then by Zorn's Lemma, there exists an ideal $I$ of $B$ such that 
\begin{itemize}
\item $I\cap A =\brac{0}$
\item If $J\supset I$ is an ideal of $B$ and if $J\cap A = \brac{0}$ then $J=I$
\end{itemize}
This ideal then makes the canonical ring homomorphism $A\rightarrow B/I$ an essential extension of $A$ (therefore $B/I$ can also be 
considered as an over-ring of $A$). For such an $I$ we say that \emph{$I$ makes $A\rightarrow B \rightarrow B/I$ an essential extension (of $A$)}
\end{defi}

\begin{prop}\label{rcrs_to_epof}
If $A$ is a real von Neumann regular ring and if $B$ is a real closure $*$ of $A$, then $B$ is actually the integral closure of $A$ in $R/I$ 
for some product of real closed fields $R:=\prod_{\spec A} K_\p$ with $K_\p$ algebraic over $A/\p$ , $(A/\p)^+\subset K_\p^2$ 
and some ideal $I$ of $R$ making 
$$A\hookrightarrow R \rightarrow R/I$$
an essential extension of $A$.
\end{prop}
\begin{proof}
We know that $B$ is essential and integral over $A$ (as porings), thus by \cite{raphael} Corollary 1.10 it is a von Neumann regular 
ring and by \cite{raphael} Lemma 1.14 the canonical spectral map 
$$\phi: \spec B \rightarrow \spec A\qquad \phi(\p):=\p\cap A$$ 
is surjective. Now to make everything short we set $X:=\spec A$ and $Y:=\spec B$.

For each $x\in X$, choose one representative $y_x\in Y$ such that $\phi(y_x)=x$. We now have the following commutative diagram
$$
\begindc{\commdiag}[1]
\obj(0,60)[A]{$A$}
\obj(60,60)[B]{$B$}
\obj(120,60)[C]{$\displaystyle\prod_{y\in Y} B/y$}
\obj(120,0)[D]{$\displaystyle\prod_{x\in X} B/y_x$}
\mor{A}{B}{}[1,3]
\mor{A}{D}{$f$}[-1,3]
\mor{B}{C}{}[1,3]
\mor{C}{D}{$\pi$}[1,0]
\enddc
$$
where $\pi$ is just a projection and $f$ is just an extension of the canonical map 
$$A\lhrarrow \prod_{x\in X} A/x$$
(since $A/x\hookrightarrow B/y_x$ is a real field extension) and the other unlabelled maps are canonical injections.

Let us set $R:=\displaystyle\prod_{x\in X} B/y_x$ and by Zorn's Lemma choose ideal $I$ of $R$ making
$$
\begindc{\commdiag}[1]
\obj(0,0)[A]{$A$}
\obj(60,0)[B]{$R$}
\obj(120,0)[C]{$R/I$}
\mor{A}{B}{$f$}[1,3]
\mor{B}{C}{}
\enddc
$$
an essential extension of $A$. Then we have the following commutative diagram of porings
$$
\begindc{\commdiag}[1]
\obj(0,60)[A]{$A$}
\obj(100,60)[B]{$B$}
\obj(100,0)[C]{$R/I$}
\mor{A}{B}{}[1,3]
\mor{B}{C}{}[1,0]
\mor{A}{C}{}[1,3]
\enddc
$$
with the vertical map being a monomorphism because both the diagonal and the horizontal maps are essential extensions. 
So $B$ can be considered as a sub-poring of $R/I$. In fact the $B\hookrightarrow R/I$ is an essential extension of $B$ and
$B$ is the integral closure of $A$ in $R/I$ (by the real closed $*$-ness of $B$).
\end{proof}

\begin{theorem}\label{vNr_rcs_iso}
Let $A$ be a real von Neumann regular ring and let $B$ and $C$ be two real closure $*$ of $A$. Set $X:=\spec A$, $Y:=\spec B$ and
$Z:=\spec C$, then $B$ and $C$ are 
$A$-isomorphic iff 
for all $x \in X$ there exists  $y_x \in Y$ and $z_x \in Z$ such that  
$$y_x\cap A=z_x\cap A=x$$   
and that the real closed fields $B/y_x$ and $C/z_x$ are ($A/x$)-isomorphic.
\end{theorem}
\begin{proof}
"$\Rightarrow$" Clear!

"$\Leftarrow$" For each $x\in X$ we have the following commutative diagram
$$
\begindc{\commdiag}[1]
\obj(0,60)[O]{$A$}
\obj(60,60)[A]{$A/x$}
\obj(160,60)[B]{$B/y_x$}
\obj(160,0)[C]{$C/z_x$}
\mor{A}{B}{}[1,3]
\mor{B}{C}{$\cong$}[1,0]
\mor{A}{C}{}[1,3]
\mor{O}{A}{}[1,5]
\enddc
$$
We therefore also have 
$$
\begindc{\commdiag}[1]
\obj(0,80)[O]{$A$}
\obj(60,80)[A]{$\displaystyle\prod_{x\in X} A/x$}
\obj(160,80)[B]{$\displaystyle\prod_{x\in X} B/y_x$}
\obj(160,0)[C]{$\displaystyle\prod_{x\in X} C/z_x$}
\mor{A}{B}{}[1,3]
\mor{B}{C}{$\cong$}[1,0]
\mor{A}{C}{}[1,3]
\mor{O}{A}{}[1,3]
\enddc
$$
So $\prod_{x\in X} B/y_x$ and $\prod_{x\in X} C/z_x$ are $A$- and ($\prod_{x\in X} A/x$)-isomorphic and we can identify them as the same
poring, setting say 
$$R:= \prod_{x\in X} B/y_x \cong \prod_{x\in X} C/z_x$$
We have moreover the following commutative diagram
$$
\begindc{\commdiag}[1]
\obj(120,60)[A]{$A$}
\obj(170,60)[B]{$B$}
\obj(240,60)[C]{$\displaystyle\prod_{y\in Y} B/y$}
\obj(120,0)[D]{$R$}
\obj(70,60)[E]{$C$}
\obj(0,60)[F]{$\displaystyle\prod_{z\in Z} C/z$}
\mor{A}{B}{}[1,3]
\mor{A}{D}{}[-1,3]
\mor{B}{C}{}[1,3]
\mor{C}{D}{$\pi_1$}[1,5]
\mor{A}{E}{}[-1,3]
\mor{E}{F}{}[1,3]
\mor{F}{D}{$\pi_2$}[-1,5]
\enddc
$$
, where $\pi_1$ and $\pi_2$ are just projections. Now by Zorn's Lemma we can find an ideal $I$ of $R$ making 
$A\hookrightarrow R\thrarrow R/I$ an essential extension of $A$.
So we have the following commutative diagram 
$$
\begindc{\commdiag}[1]
\obj(0,60)[A]{$A$}
\obj(100,60)[B]{$B$}
\obj(0,0)[C]{$C$}
\obj(100,0)[D]{$R$}
\obj(170,0)[E]{$R/I$}
\mor{A}{B}{}[1,3]
\mor{B}{D}{}[1,0]
\mor{A}{C}{}[-1,3]
\mor{C}{D}{}[1,0]
\mor{A}{D}{}[1,3]
\mor{D}{E}{}[1,5]
\enddc
$$
which becomes ($B\rightarrow R/I$ and $C\rightarrow R/I$ become monomorphisms, as $B$ and $C$  are essential extensions of $A$)
$$
\begindc{\commdiag}[1]
\obj(0,60)[A]{$A$}
\obj(100,60)[B]{$B$}
\obj(0,0)[C]{$C$}
\obj(100,0)[D]{$R/I$}
\mor{A}{B}{}[1,3]
\mor{B}{D}{}[1,3]
\mor{A}{C}{}[-1,3]
\mor{C}{D}{}[1,3]
\mor{A}{D}{}[1,3]
\enddc
$$
Now the above diagram of porings have all its morphisms as essential extensions. Moreover, recall that $B$ and $C$ are 
real closed $*$ rings and integral extensions of $A$. $R/I$ is a real ring and an essential extension of both $B$
and $C$, thus both $B$ and $C$ are integrally closed in $R/I$ (by the real closed $*$-ness of $B$ and $C$). Thus they are $A$-isomorphic to
the integral closure of $A$ in $R/I$. Note that here, we hardly concerned ourselves in checking whether the partial ordering matched, as $B$ 
and $C$ are real closed rings (by \cite{Capco} Proposition 14) and so their only partial ordering would be the squares of their elements
 (see \cite{Schw1} Proposition 2.3.10, they are also $f$-rings which means that their partial orderings cannot be strengthened, see
for instance \cite{SM} Proposition 1.11).
\end{proof}

\begin{defi} $\bullet$ Henceforth, when we say a real field is a \emph{real field extension} of another, we mean that the embedding is a poring morphism (i.e. the
 partial ordering of the lower field is also extended).
\item[$\bullet$] Let $A$ be a reduced poring, then an over-poring $B$ of $A$ is called an \emph{epof} 
(extension of product of fields) of $A$ if 
$$ B=\prod_{\p\in\spec A} K_\p\qquad \textrm{with }B^+=B^2$$
where $K_\p$ are real closed fields algebraic over $\qf{A/\p}$ and extending the partial ordering of $A^+/\p$
(i.e. $A^+/\p\subset K_\p^2$).
\end{defi}

\begin{nota}
If $A$ is a poring, then for any prime cone $\alpha\in\sper A$ we can associate it (bijectively) to a real ideal $\p=\supp(\alpha)$
and a real closed field $T$ (see \cite{KS} Definition 1a) which is algebraic over $\qf{A/\p}$ and extends the partial ordering $A^+/\p$, 
this real closed field will be denoted by $\rho(\alpha)$.
\end{nota}

\begin{cor} \label{Bvnr_unique}
Let $A$ be a real Baer von Neumann regular ring, then $A$ has a unique real closure $*$ iff there is an $X\subset\sper A$ 
such that for any epof $B$ of $A$ and any ideal $I$ of $B$ making $A\hookrightarrow B\rightarrow B/I$ an essential extension we have
\begin{equation}\label{vnr_eq} 
X=\{\alpha\in\sper A : \rho(\alpha)\cong_{A(\alpha)} C(\alpha,\p) \textrm{ for some } \p\in V_B(I+\supp(\alpha))\}
\end{equation}
where $A(\alpha):=A/\supp(\alpha)$ and $C(\alpha,\p):=\ic(A/\supp(\alpha),B/\p)$.
\end{cor}
\begin{proof}
"$\Rightarrow$" Let $C$ be the unique real closure $*$ of $A$ and set 
$$X:=\{\alpha\in \sper A \sep \rho(\alpha)\cong_{A(\alpha)} C/\p \textrm{ for some } \p\in V_C(\supp(\alpha))\}$$
Let $B$ be an epof of $A$ and $I$ be an ideal of $B$ making $A\rightarrow B/I$ an essential extension of $A$. We claim that 
$X=S$, where $S\subset \sper A$ is the right hand side of the equality in Equation (\ref{vnr_eq}).

"$\subset$" Suppose $\alpha\in\sper A$, $\p\in V_C(\supp(\alpha))$ and $\rho(\alpha)\cong_{A(\alpha)} C/\p$, then because 
$\ic(A,B/I)$ is a real closure $*$ of $A$ (since $A$ is Baer, and by \cite{raphael} Corollary 1.13 $B/I$ is also Baer and one can then 
use \cite{Capco} Theorem 15 and Proposition \ref{ic-inrc}), we may consider $C$ as an intermediate poring of $A$ and $B/I$. Thus 
$$A\subset C\subset B/I$$
Let $\mathfrak q\in V_B(I+\supp(\alpha))$ such that $(\mathfrak q +I)\cap C=\p$ (such a $\mathfrak q$ exist, because $A\hookrightarrow B/I$ is an 
essential extension). Then 
$$A/\supp(\alpha)\subset C/\p\subset B/\mathfrak q$$
are embedding of fields. Now since $C/\p$ and $B/\mathfrak q$ are real closed fields (recall $C$ is real closed $*$) and since $C/\p$ is algebraic
over $A/\supp(\alpha)$, we can write $C/\p=C(\alpha,\p)$.

"$\supset$" Let $\alpha\in \sper A$ and $\rho(\alpha)\cong_{A(\alpha)} C(\alpha,\p)$ for some $\p\in V_B(I+\supp(\alpha))$, then as before
we can write $A\subset C\subset B/I$  and $A/\supp(\alpha)\subset C/(\p\cap C) \subset B/\p$. Also we see that 
$\p\cap C\in V_C(\supp(\alpha))$ and $C/(\p\cap C)=C(\alpha,\p)$.

\vspace{5mm}
"$\Leftarrow$" This side of the proof will be done in steps. 
\begin{enumerate}
\item Let $A_i$, $i=1,2$ be real closure $*$ of $A$. From now on the variable $i$ varies always between $1$ and $2$.

\item By Proposition \ref{rcrs_to_epof}, there exists epof $R_i$ of $A_i$, and ideal $I_i$ of $R_i$ making 
$A\rightarrow R_i/I_i$ an essential extension of $A$ and such that 
$$A\lhrarrow A_i \lhrarrow R_i/I_i$$
are embedding of porings.

\item \label{Xs} Set 
$$X_i:=\{\alpha\in \sper A\sep \rho(\alpha)\cong_{A(\alpha)} \ic(A(\alpha),R_i/\p) \textrm{ for some } \p\in V_{R_i}(I_i+\supp(\alpha))\} $$
then by assumption $X_1=X_2$

\item \label{alpha_p} Note: 
	\begin{enumerate}
	
	\item \label{alpha} For all $\p\in V_{R_i}(I_i)$ there always exists an $\alpha\in \sper A$ such that $\p\cap A=\supp(\alpha)$ and
	$$\rho(\alpha)\cong_{A(\alpha)} \ic(A(\alpha),R_i/\p)$$
	This is by the very definition of the real spectrum of $A$ and the fact that $R_i/\p$ is a real closed field.
	
	\item \label{p} Because of the following composition of essential extensions
	$$A\lhrarrow A_i \lhrarrow R_i/I_i$$
	we can for each $\p\in \sper A$ find a $\p_i\in \spec A_i$ lying over $\p$ and a $\mathfrak q_i\in \spec (R_i/I_i)$ lying over $\p_i$.
	Here we didn't use essential extension, or that $A$ is Baer. Only the fact that the overrings  
	,of $A$ which is a von Neumann regular ring, are von Neuman regular (see \cite{raphael} Lemma 1.14).
	\end{enumerate}

\item So suppose $\p\in\spec A$ 
	\begin{enumerate}

	\item \label{A1} By (\ref{p}) there is a $\p_1\in\spec A_1$ and such that $\p_1\cap A=\p$ and (\ref{alpha}) there is an 
	$\alpha\in\sper A$ such that $\p=\supp(\alpha)$ and that 
	$$\rho(\alpha)\cong_{A(\alpha)} \ic(A(\alpha),R_1/\mathfrak q_1)$$
where $\mathfrak q_1\in V_{R_1}(I_1)$  with $\mathfrak q_1+I_1\in \spec (R_1/I_1)$ lying over $\p_1$ (existence of such a $\mathfrak q_1$ is by
(\ref{p})). Thus 
	$$A_1/\p_1\cong_{A(\alpha)} \ic(A(\alpha),R_1/\mathfrak q_1)$$
since 
	$$ A/\supp(\alpha)\lhrarrow A_1/\p_1 \lhrarrow R_1/\mathfrak q_1$$ 
	are real extensions of fields and $A_1/\p_1$ is algebraic over $A/\supp(\alpha)$ and itself real closed (as $A_1$ is real closed $*$).
	
	\item Now $\alpha$ in (\ref{A1}) is in $X_1$, so by (\ref{Xs}) $\alpha\in X_2$ as well and therefore 
	$$\exists\mathfrak q_2\in V_{R_2}(I_2+\supp(\alpha)) \st \rho(\alpha)\cong_{A(\alpha)} \ic(A(\alpha),R_2/\mathfrak q_2)$$
	If we set $\p_2=\mathfrak q_2\cap A_2$ we have real field extensions
	$$A/\supp(\alpha)\lhrarrow A_2/\p_2 \lhrarrow R_2/\mathfrak q_2$$
and therefore $\rho(\alpha)\cong_{A(\alpha)} A_2/\p_2$, which also means $A_1/\p_1\cong_{A/\p} A_2/\p_2$. 
Now one uses Theorem \ref{vNr_rcs_iso} to show that $A_1$ and $A_2$ are $A$-isomorphic.
	\end{enumerate}
\end{enumerate}
\end{proof}

\begin{theorem}
Corollary \ref{Bvnr_unique} holds even if we drop the assumption that $A$ is Baer.
\end{theorem}
\begin{proof}
The proof is just a slight modification of Corollary \ref{Bvnr_unique}. The sufficiency of the Corollary was proved without 
making use of the assumption that $A$ is Baer, therefore we need only show the necessity. 

Again we let $C$ be the unique real closure $*$ of $A$ and set 
$$X:=\{\alpha\in \sper A \sep \rho(\alpha)\cong_{A(\alpha)} C/\p \textrm{ for some } \p\in V_C(\supp(\alpha))\}$$
we then show that if $B$ is an epof of $A$ and $I$ is an ideal of $B$ making $A\rightarrow B/I$ an essential extension
then 
$$X=\{\alpha\in\sper A : \rho(\alpha)\cong_{A(\alpha)} C(\alpha,\p) \textrm{ for some } \p\in V_B(I+\supp(\alpha))\}$$

\vspace{5mm}
"$\subset$" Let $\alpha\in\sper A$ and $\rho(\alpha)\cong_{A(\alpha)} C/\p$ for some $\p\in V_C(\supp(\alpha))$ then we 
have the following commutative diagram of porings
$$
\begindc{\commdiag}[1]
\obj(0,35)[A]{$A/\supp(\alpha)$}
\obj(95,70)[B]{$C/\p$}
\obj(95,0)[C]{$B/\mathfrak q$}
\obj(185,35)[D]{$D/\tilde \p$}
\mor{A}{B}{}[1,3]
\mor{B}{D}{}[1,3]
\mor{A}{C}{}[1,3]
\mor{C}{D}{}[1,3]
\enddc
$$
where $D=B(B/I)$ and $\tilde \p$ is a prime ideal of $D$ lying over $\p\in\spec C$ (this is possible because of \cite{raphael}
 Lemma 1.14), and $\mathfrak q\in\spec B$ such that $\mathfrak q+I=\tilde \p \cap B/I$.

Now $B$ and $D$ are real closed von Neumann regular rings (see \cite{Capco} Remark 1 and Corollary 4) thus 
$B/\mathfrak q$ and $D/\p$ are
real closed fields (\cite{Capco} Remark 1) with $B/\mathfrak q$ being integrally closed (or algebraically closed, as we are dealing with fields)
in $D/\p$. $C$ is itself a real closed ring integral over $A$, thus $C\p$ is a real closed field integral/algebraic over 
$A/\supp(\alpha)$, therefore we can as well conclude that 
$$\ic(A/\supp(\alpha),D/\tilde\p)=\ic(A/\supp(\alpha), B/\mathfrak q)=C/\p$$
and by our choice of $\mathfrak q$ we have $\mathfrak q\in V_B(I+\supp(\alpha))$.

\vspace{5mm}
"$\supset$" Let $\alpha\in\sper A$ and $\rho(\alpha)\cong_{A(\alpha)} C(\alpha,\p)$ for some $\p\in V_B(I+\supp(\alpha))$,
 then there exists $\tilde\p \in \spec D$ lying over $\p+I$. Thus we have the following commutative diagram of porings 
$$
\begindc{\commdiag}[1]
\obj(0,35)[A]{$A/\supp(\alpha)$}
\obj(95,70)[B]{$B/\p$}
\obj(95,0)[C]{$C/\mathfrak q$}
\obj(185,35)[D]{$D/\tilde \p$}
\mor{A}{B}{}[1,3]
\mor{B}{D}{}[1,3]
\mor{A}{C}{}[1,3]
\mor{C}{D}{}[1,3]
\enddc
$$
where $\mathfrak q = \tilde p \cap C$. Again however we have (similar arguments as in "$\subset$") 
$$C/\mathfrak q = \ic(A/\supp(\alpha), D/\tilde\p)=\ic(A/\supp(\alpha,B/\p)=C(\alpha,\p)$$
and by the choice of $\tilde\p$ we have $\mathfrak q\in V_C(\supp(\alpha))$.
\end{proof}

\begin{nota}
Supposing that a commutative ring $A$ is the product of rings, say  $A:=\prod_{x\in X} R_x$, then for any element $a\in A$ and
any $x\in X$ we let $a(x)$ to be the canonical \emph{projection} of $a$ in $R_x$.
\end{nota}

\begin{defi}
If $X$ is a topological space, we say $x\in X$ is \emph{isolated} iff $\{x\}$ is a clopen set. If $x\in X$ is not isolated then we say 
it is \emph{non-isolated}.
\end{defi}

\begin{prop}\label{vNrB_isol}
Let $A$ be a real von Neumann regular Baer ring and suppose that there exists a $\p\in\spec A$ such that 
$\sharp \supp^{-1}\{\p\} > 1$ and $\p$ is isolated in $\spec A$. Then $A$ does not have a unique real closure $*$. 
\end{prop}
\begin{proof}
Since $\p$ is isolated in $\spec A$ and $A$ is von Neumann regular, there exists an $e\in A$ 
such that $e^2=e$ and $V(e)=\{\p\}$ (see for example \cite{Huckaba} Corollary 3.3(4)). Set $B:=\prod_{\mathfrak q\in \spec A} A/\mathfrak q$, then 
we may regard $A$ as a subring of $B$ and we may write
$$e(\mathfrak q)=\left\{
\begin{array}{ll}
1  & \mathfrak q\neq \p\\
0	& \mathfrak q=\p \\
\end{array}
\right.$$

Let $I\unlhd B$ be such that (Zorn's Lemma)
\begin{itemize}
\item $\brac{0}\subset I\subset eB$
\item $I\cap A=\brac{0}$
\item $\forall x\in eB\backslash I\quad (I+xB)\cap A\neq \brac{0}$
\end{itemize}
We claim that 
$$\forall x\in B\backslash I\quad (I+xB)\cap A\neq \brac{0}$$
We prove by contradiction, suppose there exists an $x\in B\backslash I$ such that $(I+xB)\cap A=\brac{0}$.
Then clearly $x\not\in eB$ and that $x\in (1-e)B^*$ (note that $B$ is a von Neumann regular ring, thus
$B$ can be written as the direct sum of $eB$ and $(1-e)B$). Therefore $1-e\in xB$ which implies 
$1-e\in I+xB$. But $1-e\in A$ and this contradicts our assumption. 

The above implies then that $A\rightarrow B/I$ is an essential extension.

Now consider $\alpha,\beta \in \sper A$ such that $\alpha\neq \beta$ and $\supp(\alpha)=\supp(\beta)=\p$. Moreover, consider
the following epofs of $A$ 
$$B_1 =\rho(\alpha)\times \prod_{\mathfrak q\in \spec A\wo{\p}} K_{\mathfrak q}$$
$$B_2 =\rho(\beta)\times \prod_{\mathfrak q\in \spec A\wo{\p}} K_{\mathfrak q}$$
where for $\mathfrak q\in \spec A\wo{\p}$, $K_{\mathfrak q}$ is a real closed field which is a real field extension
of $A/{\mathfrak q}$. We now throughout let $i=1,2$ and $\dots$

Pick an ideal $I_i\unlhd B_i$ (Zorn's Lemma) such that 
\begin{itemize}
\item $\brac{0}\subset I_i\subset eB_i$
\item $I_i\cap A=\brac{0}$
\item $\forall x\in eB_i\backslash I_i\quad (I_i+xB_i)\cap A\neq \brac{0}$
\end{itemize}
One then shows analogous to the case for $B$, that $I_i\unlhd B_i$ makes $A\rightarrow B_i/I_i$ an essential extension. 

Now $eB_i + I_i$ is a maximal ideal of 
$B_i/I_i$ which restricts to $\p$ in $A$. Note also that because $A$ is a Baer ring, $B_i/I_i$ (\cite{raphael} Corollary 1.13) is Baer and 
$$\spec B_i/I_i \longrightarrow \spec A$$ 
(see for instance \cite{raphael} Proposition 1.16 or Remark 1.17) is a homeomorphism. 

Now assume that $A$ had a unique real closure $*$, say $C$. Then $C$ can be considered as an intermediate poring of $A$ and
$B_i/I_i$ (Because of \cite{Capco} Theorem 15 and Proposition \ref{ic-inrc}). One thus has the canonical poring injection of Baer von Neumann regular rings 
$$A\lhrarrow C\lhrarrow B_i/I_i \quad i=1,2 $$
and canonical homeomorphism of spectral spaces
$$\spec B_i/I_i \stackrel{\simeq}{\longrightarrow}  \spec C \stackrel{\simeq}{\longrightarrow} \spec A\quad i=1,2 $$
and real field extensions 
$$A/\p \longrightarrow C/\p_0 \longrightarrow B_i/B_ie$$
where $\p_0$ is the unique (by the homeomorphism of spectral spaces above) prime ideal in $\spec C$ such that 
$\p_0\cap A=\p$, which is the same as $(B_ie + I_i)\cap C$.

But 
$$B_1/B_1e = \rho(\alpha), \quad B_2/B_2e=\rho(\beta)$$
and because $C$ is a real closed $*$ von Neumann regular ring, $C/\p_0$ is a real closed field which is an 
algebraic extension of $A/\p$ and has real field extensions $\rho(\alpha)$ and $\rho(\beta)$. But the natural field extension
between $A/\p$ and $\rho(\alpha),\rho(\beta)$ is an algebraic extension. This will only mean that 
$$C/\p_0 \cong_{A/\p} \rho(\alpha) \cong_{A/\p} \rho(\beta)$$
which is a contradiction, as we assumed $\alpha\neq \beta$. Thus we can finally conclude that $A$ cannot have a unique 
real closure $*$
\end{proof}

\begin{lemma}\label{ideal_e}
Let $A$ be a real Baer von Neumann regular ring and let $\p\in \spec A$ and $B$ an epof of $A$. Denote $e\in B$ for the idempotent
in $B$ such that 
$$e(\mathfrak q)=\left\{
\begin{array}{ll}
1  & \mathfrak q\neq \p\\
0	& \mathfrak q=\p \\
\end{array}
\right.$$
then $\p$ is an isolated point of $\spec A$ iff
$$\exists I\unlhd B \st B/I \textrm{ is an essential extension of } A \textrm{ and } I\subset eB$$
\end{lemma}
\begin{proof}
"$\Rightarrow$" The proof of this is contained in the beginning of the proof of Proposition \ref{vNrB_isol}.

\vspace{5mm}
"$\Leftarrow$" Since $I\subset eB$, $eB+I$ is a prime ideal of $B/I$. Now $eB$ is an isolated point of $\spec B$ (since $V_B(e)=\{eB\}$)
and so $eB+I$ is an isolated point of $\spec B/I$ (since by \cite{KS} p.99 Satz 7, $\spec B/I$ can be canonically regarded as a subspace 
of $\spec B$). But $A\hookrightarrow B/I$ is an essential extension and $A$ is Baer, therefore $\spec B/I$ is canonically homeomorphic
to $\spec A$ (by \cite{raphael} Remark 1.17). And so $(eB+I)\cap A=\p$ is an isolated point of $A$.
\end{proof}

\begin{exam}
Let $X$ be an extremally disconnected Stone space  with non-isolated points (such space if it has a non-isolated point, their count must 
be infinite. An example of such a space is the prime spectrum of infinite product of fields in the category of commutative
unitary rings). Let $K:=\Q(\sqrt 2)$ and consider 
$$A:=\{f:X\rightarrow K\sep \forall k\in K\quad f^{-1}(k) \textrm{ is open in } X\}$$
Also for any $x\in X$ and $S\subset A$ we set
$$S|x :=\{f(x)\sep f\in S\}\subset K$$

We shall now state some facts about $A$ and give partial proofs of them (as they are straightforward):
\begin{enumerate}
\newcounter{cnt}
\item $A$ has a canonical ring structure
\item $A$ is zero-dimensional and reduced (ie. regular) 

Take any $f\in A$, show that $g:X\rightarrow K$, defined by $g(x)=0$ iff $f(x)=0$ and $g(x)=f(x)^{-1}$ iff $f(x)\neq 0$, is in $A$ and 
in fact the quasi-inverse of $f$.
\item There is a canonical homeomorphism 
$$X\rightarrow \spec A \quad x\mapsto \p_x:=\{f\in A\sep f(x)=0\} $$
Thus $A$ is Baer (as $X$ is extremally disconnected.
\item\label{proj_prime} Because $X$ is Boolean, for any $x\in X$ we have the identity 
$$\p_x |z=\left\{
\begin{array}{ll}
K  & z\neq x\\
0	& z=x \\
\end{array}
\right.$$
\item For all $x\in X$, one has a well-defined and canonical isomorphism of real fields
$$A/\p_x \stackrel{\simeq}{\longrightarrow} K \qquad f\mod \p_x \mapsto f(x)$$
\setcounter{cnt}{\value{enumi}}\label{lastitem}
\end{enumerate}

This ring can be found in some literature, compare for instance the remark in the last paragraph of \cite{tic} p.779. For the proof of the facts
above one also sees similarities in the arguments found in \cite{AM} Chapter 1, Excercise 26.

Fix a non-isolated point $y\in X$
 and for $i\in\{0,1\}$ set $P_i$ to be \underline{the} total
ordering of $K$ such that $(-1)^{i}\sqrt 2 \in P_i$. Now set
$$A^+:=\{f\in A\sep f(x)\in P_0 \textrm{ if } x\neq y \textrm{ and } f(y)\in \sum K^2 \}$$
one easily checks that $A^+$ is a partial ordering of $A$.  
Without loss of generality we identify $X$ and $\spec A$.

We now claim:

\begin{enumerate}
\setcounter{enumi}{\value{cnt}}
\setlength{\itemindent}{-13pt}%

\item \underline{$\supp|(\sper A)\backslash\supp^{-1}\{y\}$ is injective}

For this we need to first study the real spectrum of $A$ (equiped with the partial ordering $A^+$. By definition if 
$\alpha\in \sper A$ then $A^+\subset \alpha$ and $-\alpha \cap \alpha=\supp(\alpha)\in \spec A$. Thus one easily checks 
,by the properties defining prime cones (see e.g. ADD Reference Bochnak) and by \ref{proj_prime}., that for all $x\in X$
either $\alpha|x = K$ (i.e. $\supp(\alpha)|x=K$) or that $\alpha|x$ is a total ordering of $K$ (i.e. $\supp(\alpha)|x=\{0\}$).

Since $A$ is von Neumann regular, every $x\in X\wo{y}$ has a $\supp$ preimage (i.e. $\supp$ is surjective), say 
$\alpha\in \sper A$ is a point in the preimage of $x$. Now for any $z\in X$, $A^+|z\subset \alpha |z$ so by the argument
in the previous paragraph $\alpha$ must have the property (because $A^+|x=P_0\subset \alpha |x$ and $\alpha|x$ is a total
ordering of $K$)
$$\alpha |z=\left\{
\begin{array}{ll}
K  & z\neq x\\
0	& z=x \\
\end{array}
\right.$$
this alone completely determines $\alpha$ (by the very definition of $A$ and the fact that $X$ is Boolean). Thus the claim 
holds.

\item \underline{$\sharp\supp^{-1}\{y\} = 2$}

If $\alpha\in\supp^{-1}\{y\}$ then this time $\supp(\alpha)|y =\{0\}$ implies that 
$\alpha|y$ is a total ordering of $K$ and 
$$A^+|y=\sum K^2 \subset \alpha|y$$
giving $\alpha|y$ the choice of being either 
$P_0$ or $P_1$ (and $\alpha|x = 0$ for $x\in X\wo{y}$). Thus giving use $2$ points in $\supp^{-1}\{y\}$.

\item \underline{The poring $A$ has a unique real closure $*$}

Let $C$ be a real closure $*$ of $A$  (so $C^+\supset A^+$), then by Proposition \ref{rcrs_to_epof} there exists an epof 
$B=\prod_{x\in X} K_x$ of $A$ and an ideal $I\unlhd B$ such that $A\rightarrow B/I$ is an essential extension of $A$ and 
$C=\ic(A,B/I)$. One
then has the commutative diagram of porings
$$
\begindc{\commdiag}[1]
\obj(0,60)[A]{$A$}
\obj(100,60)[B]{$B/I$}
\obj(0,0)[C]{$A/x$}
\obj(100,0)[D]{$B/\tilde x$}
\obj(180,30)[E]{$x\in X,\, \tilde x\in V_B(I+x)$}
\mor{A}{B}{}[1,3]
\mor{B}{D}{$\phi_{2}^x$}[1,5]
\mor{A}{C}{$\phi_{1}^x$}[-1,5]
\mor{C}{D}{}[1,3]
\enddc
$$
with all the maps being canonical (note that $\sharp V_B(I+x)=1$ because $A$ is Baer and $B/I$ an essential extension of $A$).

Let $f\in A$ such that $f(x)=\sqrt 2$ for all $x\in X$, and set 
$$\pi :\prod_{x\in X} K_x \longrightarrow \prod_{x\in X\wo{y}} K_x$$
to be the canonical (poring) projection of $B$. Observe now that from the choice of $A^+$, one has $\pi(f)\in \pi(B)^+$.

By Lemma \ref{ideal_e} $I\not\subset eB$ where $e\in B$ is defined by 
$$e(x)=\left\{
\begin{array}{ll}
1  & x\neq y\\
0	& x=y \\
\end{array}
\right.$$
this implies at once that $1-e\in I$ and so there is a well-defined isomorphism of porings 
$$B/I \longrightarrow \pi(B)/\pi(I) \qquad b+I\mapsto \pi(b)+\pi(I)$$
Note that the working rings are real von Neumann regular, therefore $I$ and $\pi(I)$ are real ideals. 

All in all, one can conclude that (by identifying $B/I$ with $\pi(B)/\pi(I)$) 
$$\sqrt 2= \phi_1^x (f)=\phi_2^x(\pi(f))\in (B/\tilde x)^+ \qquad x\in X$$

Now $\spec C,\spec A, \spec (B/I)$ are canonically homeomorphic (spectral maps of essential extensions of Baer regular rings). 
Also 
$C/xC =\ic (A/x,B/\tilde x)$ for all $x\in X$ (note that $xC\in \spec C$, 
for an idea why see \cite{raphael} proof of Lemma 1.14 and Remark 1.17, 
and is the unique prime ideal of $C$ lying over $x$). And because $\sqrt 2\in (C/xB) \cap (B/\tilde x)^+$ we conclude that 
$C/xB \cong_{A/x} R_0$ where $R_0$ is the field of real algebraic numbers. 

The choice of $x\in X$ and $C$ as the real closure $*$ of $A$ was
arbitrary, therefore one can say that family of factor fields (indexed by $X$) of any two real closure $*$ of $A$ are the same (upto 
suitable $A/x$-isomorphism for some $x\in X$). With Theorem \ref{vNr_rcs_iso} we can therefore conclude that $A$ has a unique 
real closure $*$.

\comment{
\item There is an $X\subset\sper A$ 
such that for any epof $B$ of $A$ and any ideal $I$ of $B$ making $A\hookrightarrow B\rightarrow B/I$ an essential extension we have
$$
X=\{\alpha\in\sper A : \rho(\alpha)\cong_{A(\alpha)} C(\alpha,\p) \textrm{ for some } \p\in V_B(I+\supp(\alpha))\}
$$
where $A(\alpha):=A/\supp(\alpha)$ and $C(\alpha,\p):=\ic(A/\supp(\alpha),B/\p)$.
}
\end{enumerate}
\end{exam}

The example above just shows us that there are von Neumann regular rings with unique real closure $*$ and yet they are not $f$-rings. In 
fact we have a general case below

\begin{theorem}
Let $A$ be a real von Neumann regular ring and 
$$X:=\{\p\in \spec A \sep \sharp\supp^{-1}(\p)=1\}$$
be dense in $\spec A$, then $A$ has a unique real closure $*$
\end{theorem}
\begin{proof}
Let $C_1$ and $C_2$ be two real closure $*$ of $A$ and let $i\in\{1,2\}$. Denote $Z_i=\spec C_i$, then the canonical morphism 
$$ Z_i \twoheadrightarrow \spec A \qquad z\mapsto z\cap A$$
is surjective. 

So for each $x\in X$ fix $z_i^x \in Z_i$  such that $z_i^x\cap A=x$. We then have the canonical poring morphism 
$$A\longrightarrow C_i \longrightarrow  \prod_{x\in X} C_i/z_i^x$$
which is a monomorphism (i.e. $A$ is embedded as poring in $\prod_{x\in X} C_i/z_i^x$) since $X$ is dense in $\spec A$.

Moreover 
$$\prod_{x\in X} C_1/z_2^x = \prod_{x\in X} C_1/z_2^x$$
Because for all $x\in X$, $C_i/z_i^x$ is a real closed field algebraic over $A/x$ and a real field extension of it,  but $A/x$ 
has only one (upto $A/x$-isomorphism) real field extension which is a real closed field (by the very definition of $X$).
So set 
$$B:=\prod_{x\in X} C_1/z_1^x$$
we then have the following commutative diagram of porings 
$$
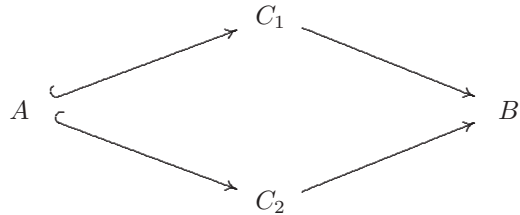
\begindc{\commdiag}[1]
\obj(0,35)[A]{$A$}
\obj(95,70)[B]{$C_1$}
\obj(95,0)[C]{$C_2$}
\obj(185,35)[D]{$B$}
\mor{A}{B}{}[1,3]
\mor{B}{D}{}
\mor{A}{C}{}[1,3]
\mor{C}{D}{}
\enddc
$$
Now by Zorn's Lemma there is an $I\unlhd B$ such that 
$$A\lhrarrow B\twoheadrightarrow B/I$$
is an essential extension of $A$. Furthermore we learn that $C_i \rightarrow B/I$ itself is an essential extension (of $C_i$) 
as $A\hookrightarrow C_i$ and $A\rightarrow B$ are essential extensions (of $A$).  We thus now have the commutative diagram 
of porings below 
$$
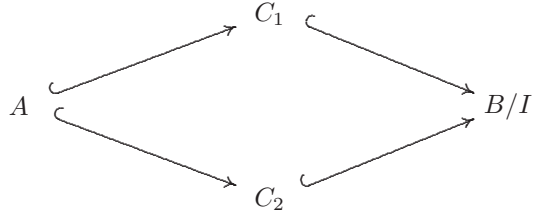
\begindc{\commdiag}[1]
\obj(0,35)[A]{$A$}
\obj(95,70)[B]{$C_1$}
\obj(95,0)[C]{$C_2$}
\obj(185,35)[D]{$B/I$}
\mor{A}{B}{}[1,3]
\mor{B}{D}{}[1,3]
\mor{A}{C}{}[1,3]
\mor{C}{D}{}[1,3]
\enddc
$$
We then finally see that (by the fact that $C_i$ is a real closure $*$ of $A$) 
$$C_i\cong_A \ic(A,B/I) \qquad i=1,2$$
which means that
$$C_1\cong_A C_2$$
\end{proof}

Finally we give a characterisation of real Baer von Neumann regular rings having unique real closure $*$.

\begin{nota}
Given a poring $A$ and $a_1,\dots,a_n\in A$ we use the notations
\begin{eqnarray*}
P_A(a_1,\dots,a_n) & :=  & \{\alpha\in\sper A \sep a_1,\dots,a_n\in \alpha\backslash\supp_A(\alpha)\} \\
& = &\{\alpha\in\sper A \sep a_1(\alpha),\dots,a_n(\alpha)>0\}
\end{eqnarray*}
where by $a_i(\alpha)$ we mean $a_i \mod \supp(\alpha)$ considered as an element of $\rho(\alpha)$.

If it is clear with which ring we are working with, we drop the subscript $A$ above and we just write 
$P(a_1,\dots,a_n)$.
\end{nota}

\begin{theorem}
Let $A$ be a formally real Baer, von Neumann regular ring. Then $A$ has no unique real closure $*$ iff there exists an $x\in A$ such that
$$[\supp P(x)\cap\supp P(-x)]^\circ \neq \emptyset$$
\end{theorem}
\begin{proof}
"$\Rightarrow$" Let $C_1$ and $C_2$ be two real closure $*$ of $A$ such that they are not $A$-isomorphic. 
Then by Theorem \ref{vNr_rcs_iso} there exists a $\p\in\spec A$ such that 
$$C_1/\p C_1\not\cong_{A/\p} C_2/\p C_2$$
note that $\p C_i\in \spec C_i$ is a unique prime ideal in $\spec C_i$ lying over $\p\in\spec A$ 
(here we used the fact that $A$ is Baer, see \cite{raphael} proof of Proposition 1.16).  Thus there are 
distinct $\alpha_1,\alpha_2\in \supp^{-1}(\p)$ such that 
$$C_i/\p C_i \cong_{A/\p} \rho(\alpha_i)\qquad i=1,2$$

Let now $i=1,2$, we then have the following commutative diagram of spectral spaces
$$
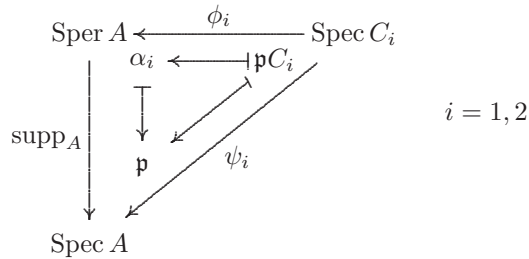
\begindc{\commdiag}[1]
\obj(0,60)[A0]{$\sper A$}
\obj(20,50)[A]{$\alpha_i$}
\obj(100,60)[B0]{$\spec C_i$}
\obj(70,50)[B]{$\p C_i$}
\obj(0,-20)[C0]{$\spec A$}
\obj(20,10)[C]{$\p$}
\obj(150,30){$i=1,2$}
\mor{B0}{A0}{$\phi_i$}[-1,0]
\mor{B}{A}{}[1,4]
\mor{B0}{C0}{$\psi_i$}
\mor{B}{C}{}[1,4]
\mor{A0}{C0}{$\supp_A$}[-1,0]
\mor{A}{C}{}[1,4]
\enddc
$$
where $\phi_i$ is the canonical map between the real spectrum of $C_i$ and the real spectrum of $A$ after identifying $\spec C_i$ and $\sper C_i$ (as $\supp_{C_i}$ 
is a homeomorphism), and $\psi_i$ is the canonical map between the prime spectrum of $C_i$ and that of $A$ (by contracting the prime ideal 
of $C_i$ to $A$). By \cite{raphael} Remark 1.17 (here $A$ Baer is used), $\psi_i$ is a homeomorphism of spectral spaces. 

Choose now $x\in\alpha_1\backslash \alpha_2$. Then we have 
$$\supp_A P(x)\supset \psi_1(\phi_1^{-1}P(x))$$ 
and $\alpha_1\in P(x)$ implies that $\p\in\psi_1(\phi_1^{-1}P(x))$. Similarly 
$$\supp_A P(-x)\supset \psi_2(\phi_2^{-1}P(-x))$$ 
with $\alpha_2\in P(-x)$. Thus 
$$\psi_1(\phi_1^{-1}P(x))\cap\psi_2(\phi_2^{-1}P(-x))$$
is a nonempty open subset ($\p$ is contained in it) 
of 
$$\supp P(x)\cap\supp P(-x)$$

\vspace{5mm}
"$\Leftarrow$" Let $x\in A$ such that there exists an nonempty open set 
$$U\subset \supp P(x)\cap\supp P(-x)$$
Without loss of generality we may assume that $U$ is clopen (as $\spec A$ is a Stone space). Define 
$$f:\spec A \longrightarrow \dot{\bigcup_{\p\in\spec A}} A/\p $$
by 
$$f(\p)=\left\{
\begin{array}{ll}
x\mod\p  & \p\in U\\
0	& \p\not\in U \\
\end{array}
\right.$$
Since $A$ is a von Neumann regular ring, we can identify $A_\p$ and $A/\p$ (i.e. they are canonically isomorphic as fields) and regard 
$f$ as an element of the global section of the (canonical) sheaf structure of $A$. Thus we may regard $f$ as an element of $A$ and 
identify $f\mod \p$ with the germ $f(\p)$ in the stalk $A_\p\cong A/\p$.

Choose the following epofs of $A$ 
$$B_1 = \prod_{\p\not\in U} \rho(\gamma_\p) \times \prod_{\p\in U} \rho(\alpha_\p)$$
$$B_2 = \prod_{\p\not\in U} \rho(\gamma_\p) \times \prod_{\p\in U} \rho(\beta_\p)$$
where for $\p\not\in U$, $\gamma_\p\in\supp^{-1}(\p)$ (note that $\supp_A$ here is surjective) and for $\p\in U$
$$\alpha_p,\beta_\p\in\supp^{-1}(\p)$$
is chosen in such a way that $\alpha_\p\in P(x)$ and $\beta_\p\in P(-x)$.

One now observes at once that $f>0$ in $B_1$ and $f<0$ in $B_2$. Moreover, we have the following commutative diagram of 
porings
$$
\begindc{\commdiag}[1]
\obj(0,60)[A]{$A$}
\obj(50,60)[B]{$B_i$}
\obj(100,60)[B2]{$B_i/I_i$}
\obj(0,0)[C]{$C_i$}
\obj(100,0)[D]{$B(B_i/I_i)$}
\obj(180,30)[E]{$i=1,2$}
\mor{A}{B}{}[1,3]
\mor{B}{B2}{}[1,5]
\mor{B2}{D}{}[1,3]
\mor{A}{C}{}[1,3]
\mor{C}{D}{}[1,3]
\enddc
$$
where the right vertical map is the canonical morphism (taking the factor module $I_i$) and $I_i$ is an ideal of $B_i$ making 
$A\rightarrow B_i/I_i$ an essential extension of $A_i$, and $C_i:=\ic(A,B(B_i/I_i))$ (which we know, by Proposition  \ref{ic-inrc} and
\cite{Capco}  Theorem 16, to be a real closure $*$ of $A$). $A$ can be regarded as a subring of $B_i/I_i$ and we know in particular that 
$$f\in (B_1/I_1)^+\wo{0}\Rightarrow f\in B(B_1/I_1)^+\wo{0} \textrm{ and } f\in -(B_2/I_2)^+\wo{0} \Rightarrow f\in -B(B_2/I_2)^+\wo{0}$$ 
Thus
$$f\in C_1^+\wo{0} \textrm{ and } f\in -C_2^+\wo{0}$$
making us conclude that 
$$C_1\not\cong_A C_2$$
\end{proof}

\begin{ack}
I would like to thank Prof. Niels Schwartz for his most valuable advises.
\end{ack}

\end{document}